\documentclass[12pt, a4paper]{amsart}
\usepackage{amsfonts,amsmath, amsthm, amssymb, amscd}
\usepackage{latexsym}

\input{amssym.def}
\input{amssym}

\usepackage[dvips]{graphicx}
\usepackage{hyperref}
\usepackage{enumitem}

\newtheorem{defi}{Definition}[section]
\newtheorem{satz}[defi]{Theorem}
\newtheorem{prop}[defi]{Proposition}

\newtheorem{lemma}[defi]{Lemma}
\newtheorem{bem}[defi]{Remark}
\newtheorem{folg}[defi]{Corollary}

\newcommand {\R}{\mathbb{R}} 
 
\newcommand {\C}{\mathbb{C}}

\DeclareMathOperator{\vol}{vol}

\DeclareMathOperator{\grad}{grad}

\begin{document}
\title{uncertainty principles on Harmonic Manifolds of Rank One}

\author{Oliver Brammen}
\date{\today}
\address{Faculty of Mathematics,
Ruhr University Bochum, 44780 Bochum, Germany}
\email{oliver.brammen@rub.de}
\thanks{Funded by the Deutsche Forschungsgemeinschaft (DFG, German Research Foundation) – Project-ID 281071066 – TRR 191.}

\begin{abstract}
We show various uncertainty principles for the Fourier transform on harmonic manifolds of rank one. In particular, we derive a Heisenberg uncertainty principle, a Morgen theorem, an uncertainty principle for the Schrödinger equation and a version of Hömanders theorem. Furthermore, we generalise the in the Euclidean case well-known Hausdorff-Young inequality for the Fourier transform, to harmonic manifolds of rank one. 
\end{abstract}

\maketitle

\section{Introduction}
A classical result going back to the 1920s about the Fourier transform for $L^2$ functions on the real line is the Heisenberg uncertainty principle, for $f\in L^2(\R)$:
\begin{align*}
\int_{\R}x^2\lvert f(x)\rvert^2\,dx\cdot \int_{\R}\lambda^2\lvert \mathcal{F}(f)(\lambda)\rvert^2 \,d\lambda\geq \frac{1}{4}\lVert f\rVert_2^4,
\end{align*}
where $\mathcal{F}(f)(\lambda)=\frac{1}{\sqrt{2\pi}}\int_{\R} f(x)e^{-\lambda x}\,dx$. 
It states that a function and its Fourier transform can not be localised at the same time. In the ensuing decades, various other uncertainty principles were discovered, that pose bounds on the decay at infinity of a function and its Fourier transform (see \cite{surveyuncertenti} and the references within).
In the '90s and early 2000'es many of those were generalised to the Helgason Fourier transform on symmetric spaces and non-symmetric Damek-Ricci spaces. Among those are versions of the Hardy and Morgan theorem (\cite{uncertaintysym} \cite{bagchiray1998}).
This covers all known examples of non-compact harmonic manifolds. These are manifolds where the volume growth of a geodesic ball only depends on its radius. Starting with the work of Pyerimhoff and Samiou \cite{PS15} it became apparent that some of the analytic properties that are present in Damek-Ricci spaces due to their algebraic structure can be derived only using geometric means. The authors in \cite{Biswas2019} developed a Helgason Fourier transform on harmonic manifolds under the assumption of purely exponential volume growth. 
This together with the result from \cite{Heisenberg}, which gives a Heisenberg uncertainty principle for radial functions on harmonic manifolds of purely exponential volume growth, motivates the main part of this article dealing with the generalisation of various uncertainty principles together with some in the Euclidean case well know inequalities for the Fourier transform to harmonic manifolds with purely exponential volume growth. 
In section \ref{s7} we derive a Heisenberg uncertainty principle for the Fourier transform on harmonic manifolds with purely exponential volume growth, using the bounds on the $L^2$ norm of the heat kernel derived from the correspondence to hypergroups. 
Section \ref{s8} deals with generalisation of the Morgen theorem on rank one symmetric spaces to harmonic manifolds with pinched negative curvature using the Radon transform established in  \cite{Rouvire2021} (see section \ref{s5} for the definition and some important properties) and bounds on the radial eigenfunctions of the Laplace operator (section \ref{s4}) derived from the fact that the density function defines a hypergroup structure on the real line. 
Using an analogous approach, we derive an uncertainty principle for the Schrödinger equation on harmonic manifolds of pinched negative curvature in section \ref{schrö}. 

Furthermore,  in section \ref{HY}, we apply the bounds on the radial eigenfunctions of the Laplace operator and interpolation arguments to generalise the Hausdorff-Young inequality. Note that there are versions for symmetric spaces in \cite{interpolation2}. In section \ref{s10} we then show a version of the Heomander theorem for harmonic manifolds of purely exponential volume growth this generalises results from \cite{MR2679708}. 
\section{Preliminaries}\label{s1}
\subsection{Notation}
Throughout this article, we always assume the manifold $(X,g)$  to be a Riemannian,  complete, connected, without conjugated points and smooth. Furthermore, when not explicitly stated otherwise we will denote for a function $f:X\to \C$ by 
$\int_X f(x)\,dx$ the integral of a $f$ with respect to the volume form induced by the Riemannian metric $g$ on $X$. We will denote the space of $k$-times  differentiable functions on $X$ (with compact support) by $C^{k}(X)$ ($C^k_c(X)$) and the $L^p$-spaces with regards to the volume form by $L^p(X)$, equipped with the topology induced by the norm. We will supply the measure as the second argument in the bracket when deemed necessary. 
\subsection{Geometric Property of  Harmonic Manifolds}
In this section, we recall some facts needed in the proceeding. The main source of reference for those are the articles \cite{Biswas2019} and \cite{kreyssig2010introduction}\cite{knieper2013noncompact2}. A Riemannian manifold $(X,g)$ is called harmonic if for every point $\sigma\in X$ there exists a non-trivial radial, i.e. only on the geodesic distance $d(\sigma,x)$ dependent, solution of the Laplace equation 
\begin{align*}
f:X\to\R\\
\Delta f=0,
\end{align*}
where $\Delta=\operatorname{div}\operatorname{grad}$ is the Laplacian on $X$.  By Allamigeo's theorem a simply connected harmonic manifold has no conjugated points, hence by the Hadamard Cartan theorem the exponential map $\operatorname{exp}: T_{\sigma} X\to X$ is a diffeomorphisem for every $\sigma\in X$. An equivalent definition of $X$ being harmonic is that the density function $\sqrt{g_{ij}}$ is radial in geodesic coordinates, and therefore in this coordinate, we can see the density function $A(r)$ as the Jacobian of the map sending $v\in S_{\sigma}X$ to $\operatorname{exp}(rv)$. Now $X$ is said to be of purely exponential volume growth if there exists some constant $C\geq1$ and $\rho>0$ such that:
\begin{align}\label{eq:exponential}
\frac{1}{C}\leq \frac{A(r)}{e^{2\rho r}}\leq C.
\end{align}
This property is by \cite{knieper2009new} equivalent to
\begin{itemize}
\item Anosov Geodesic Flow
\item Gromov Hyperbolicity 
\item Rank one.
\end{itemize}
Where the rank of a manifold without conjugated points is defined in \cite{knieper2009new} which is a generalisation of the well-known notion of rank for manifolds of non-positive curvature.
It was a long-standing conjecture coined the Lichnerowicz conjecture \cite{Lich} that all harmonic manifolds are either Euclidean spaces or symmetric spaces for rank one. 
 It was confirmed for simply connected spaces by Szabo \cite{szabo1990} but shortly after in 1992  Damek and Ricci \cite{Damek_1992} provided for dimension 7 and higher a class of homogeneous harmonic spaces that are non-symmetric and have purely exponential volume growth. In dimensions 5 and lower the conjecture was confirmed by \cite{Lich}, \cite {Walker}, \cite{Besse1978}, \cite{Nik} as well as for non-compact harmonic manifolds with negative curvature admitting a compact quotient \cite{Gallot1995}.  
Ranjan and Shah \cite{Ranjan2002} showed that among the non-compact harmonic spaces, only the ones without vanishing curvature of the horosphere are not flat. 
In 2006 Heber \cite{Heber_2006} settled the conjecture in the homogeneous case showing that a non-compact simply connected homogeneous harmonic manifold is either flat, rank one symmetric or a non-symmetric Damek-Ricci space. Using the equality above Knieper \cite{knieper2009new} showed that if a harmonic manifold of rank one admits a compact quotient it is already a symmetric space of rank one. 

Let $(X,g)$ be a non-compact simply connected harmonic manifold of rank one and consider the geodesic boundary $\partial X$ consisting of equivalence classes of geodesic rays, where two rays are equivalent if their distance is bounded. We equip $\partial X$ with the cone topology. Note that this makes $\overline{X}=X\cup\partial X$ a compact space. Furthermore the Busemann function  of a geodesic ray $\gamma_v$, $v\in S_{\sigma}X$ 
\begin{align*}
b_v(x):=\lim_{t\to\infty}d(\gamma_v(t),x)-t.
\end{align*}
depends only on the equivalence class of the ray.  Hence 
for $x\in X$ and $\xi\in\partial X$  $\gamma\in \xi$ we can alternatively define the  Busemann function $B_{\xi,x}:X\to\R$ by $$B_{\xi,x}(y)=\lim_{t\to\infty}(d(y,\gamma(t))-d(x,\gamma(t)),$$ 
and we have the cocycle property 
$$B_{\xi,x}=B_{\xi,\sigma}-B_{\xi,\sigma}(x) \forall x,\sigma\in X.$$ 
For a detailed account of this see \cite{Biswas2019}.
By \cite{BusemannHarmonic} we therefore have: 
$\Delta B_{\xi,\sigma} =2\rho$, where $2\rho$ is the mean curvature of the horospheres, which is constant on $X$. 
Furthermore the authors in \cite{Knieper2016} obtained a family of probability measures $\{\mu_x\}_{x\in X}$ on the geometric boundary, which are pairwise absolutely continuous with Radon-Nikodym derivative  $\frac{d\mu_x}{d\mu_y}(\xi)=e^{-2\rho B_{\xi,x}(y)}$.

\subsection{Fourier Transform and Plancherel Theorem on Rank One Harmonic Manifolds}
Consider a $C^2$ function $f$  on $X$  and $u$ a $C^{\infty}$ function on $\R$. Then the authors in \cite{Biswas2019} calculated:
$$\Delta (u\circ f)=(u''\circ f)\lVert \operatorname{grad} f\rVert^2 +(u'\circ f)\Delta f.$$

This yields the spherical and horospherical part of the Laplacian,

\begin{align}\label{equ:radiallaplace}
\Delta(u\circ d_x)=u^{\prime\prime}\circ d_x +u^{\prime}\circ d_x\cdot \frac{A^{\prime}}{A}\circ d_x.
\end{align}

\begin{align}
\Delta( u\circ b_v)=u^{\prime\prime}\circ b_v+2\rho\cdot u^{\prime}\circ b_v.
\end{align}
Therefore we have:
 $g(y)=e^{(i\lambda-\rho)B_{\xi, x}(y)}$ is a eigenfunction of the Laplacian with $g(x)=1$ and $\Delta g= -(\lambda^2+\rho^2)$ for $\lambda\in \C$.
Using this as a starting point authors in \cite{Biswas2019} derived a Helgason Fourier transform and its inverse in the case of rank one via a careful analysis of the operator $L_A$ given by the radial part of the Laplacian: 
The operator 
\begin{align}\label{eq:A}
 L_{A}:=\frac{d^2}{dr^2}+\frac{A^{\prime}(r)}{A(r)}\frac{d}{dr}
\end{align}
is the radial part of the Laplacian on $X$. 
Let
\begin{align}
\varphi_{\lambda}:\R^+\to\R, \quad\lambda\in [0,\infty)\cup [0,i\rho]
\end{align}
be the eigenfunction of $L_A$ with
\begin{align}
L_{A}\varphi_{\lambda}=-(\lambda^2+\rho^2)\varphi_\lambda
\end{align}
and which admits a smooth extension to zero with $\varphi_{\lambda}(0)=1$. Obviously we have that for $\sigma\in X$  $\varphi_{\lambda,\sigma}=\varphi_{\lambda}\circ d_{\sigma}$ is a radial eigenfunction of $\Delta$ with eigenvalue $-(\lambda^2+\rho^2)$.
In \cite{PS15} the authors 
defined the radial Fourier transform on general non-compact harmonic manifolds: 
\begin{defi}
Let $f:X\to\R$ be, i.e. $f=u\circ d_x$ for some $x\in X$, where $u:[0,\infty)\to\R$ and $d_x:X\to \R$ is the distance function. The Fourier transform of $f$ is given by: 
\begin{align*}
\widehat{f}(\lambda)&:=\int_X f(y)\varphi_{\lambda,x}(y)\,dy\\
&=\omega_{n-1}\int_0^{\infty}u(r)\varphi_{\lambda}(r)A(r)\,dr.
\end{align*}
where $\omega_{n-1}:=\vol(S^{n-1}).$
\end{defi}
In the case of rank one the authors in \cite{Biswas2019} defined the Helgason Fourier transform by: 
\begin{defi}
		Let $\sigma\in X$ for $f:X\to \C$ measurable, the Fourier transform of $f$ based at $\sigma$ is given by 		
		$$\tilde{f}^{\sigma}(\lambda,\xi)=\int_{X}f(y)e^{(-i\lambda-\rho)B_{\xi,\sigma}(y)}\,dy$$
		for $\lambda\in\C$, $\xi\in\partial X$ for which the integral above converges and  $B_{\xi,\sigma}$ the Busemann function in direction  $\xi$ based at $\sigma$, i.e.
			$$B_{\xi,\sigma}(\sigma)=0.$$
\end{defi}
\begin{bem} The authors in \cite{Biswas2019} showed that 
    \begin{enumerate}
        \item The radial eigenfunctions have the representation \begin{align}\label{radialeigen}
			\varphi_{\lambda,x}(y):=\int_{\partial X}e^{(i\lambda-\rho)B_{\xi,x}(y)}\,d\mu_x (\xi).
   \end{align}
   \item Using (\ref{radialeigen}) on obtains that  $\hat{f}=\widetilde{f}^{\sigma}$ if $f$ is radial around $\sigma$.
   \item For $x,\sigma\in X$ we have by the cocycle property of the Busemann function: 
   \begin{align}\label{eq:fourpoint}
			\tilde{f}^x (\lambda,\xi)=e^{(i\lambda+\rho)B_{\xi,\sigma}(x)}\tilde{f}^{\sigma}(\lambda,\xi).
			\end{align}
    \end{enumerate}
\end{bem}

Under the condition of rank one we have for two linear independent solutions to $L_{A,x}u=-(\lambda^2+\rho^2)$ 
$\Phi_{\lambda}$ and $\Phi_{-\lambda}$ which are asymptotic to exponential functions ea. 
\begin{align}
\Phi_{\pm \lambda}(r)= e^{(\pm i\lambda-\rho)r}(1+o(1))\text{ as }r\to\infty
\end{align}
that
\begin{align}\label{eigende}
\varphi_{\lambda}=\mathbf{c}(\lambda)\Phi_{\lambda}+\mathbf{c}(-\lambda)\Phi_{-\lambda}\quad\forall \lambda\in \C\setminus \{0\}.
\end{align}
Where $\mathbf{c}:\C\to\C$ is a function holomorph on the lower have plane with a continuous extension to the real line, corresponding to the Harish-Chandra’s c-function in the symmetric case.

Using results form  \cite{Bloom1995TheHM}  the authors in \cite{Biswas2019} showed that there is a constant $C_0$ such that

			\begin{align}
			f(x)=C_0 \int_{0}^{\infty}\int_{\partial X}\tilde{f}^{\sigma}(\lambda,\xi)e^{(i\lambda-\rho)B_{\xi,\sigma}(x)}\,d\mu_{\sigma}(\xi)
			\vert \mathbf{c}(\lambda)\vert^{-2}\,d\lambda.
			\end{align}
	
		Additionally, they obtained  a Plancherel theorem:

\begin{satz}[\cite{Biswas2019}]\label{Placherel Theorem}
		Let $\sigma\in X$, $f,g\in C_{c}^{\infty}(X)$. Then we have:
			$$\int_X f(x)\overline{g(x)}\,dx=C_{0}\int_{0}^{\infty}\int_{\partial X}\tilde{f}^{\sigma}(\lambda,\xi)\overline{\tilde{g}
			^{\sigma}(\lambda,\xi)} \vert \mathbf{c}(\lambda)\vert^{-2}\,d\mu_{\sigma}(\xi)d\lambda 
			$$ 
		and the Fourier transform extends to an isometry between $L^2(X)$ and\\
		 $L^2((0,\infty)\times \partial X,C_0 \vert \mathbf{c}(\lambda)\vert^{-2} \,d\mu_{\sigma}(\xi)\,d\lambda).$
\end{satz}
Since for $f\in L^p(X)$ and $1\leq p<2$ it is not clear whether the Fourier transform exists we need the following lemma. 
\begin{lemma} \label{lemma:holo2}
Let $f\in L^{p}(X)$, $1\leq p\leq 2$, then $\widetilde{f}^{\sigma}(\lambda,\xi)$ exists for every $\lambda\in \C$ with $\operatorname{im}\lambda\in (-\gamma_{q}\rho,\gamma_{q}\rho)$ where $\frac{1}{q}+\frac{1}{p}=1$ and $\gamma_{q}=1-\frac{2}{q}$. Moreover $\widetilde{f}^{\sigma}$ is holomorphic on the stripe $\operatorname{im}\lambda\in (-\gamma_{q}\rho,\gamma_{q}\rho). $
\end{lemma}
\begin{proof}
\begin{align*}
\left \lvert \int_{X}f(x)e^{-(i\lambda +\rho)B_{\xi,\sigma}(x)}\,dx\right \rvert &\leq \int_{X}\lvert f(x)\rvert \lvert e^{-(i\lambda+\rho)B_{\xi,\sigma}(x)}\rvert \,dx\\
&\leq \lVert f\rVert_{p}\cdot \lVert \varphi_{it,\sigma}\rVert_{q}<\infty,
\end{align*}
where the last line is due to \cite[Lemma 8.1]{Biswas2019}. Hence the claim followed by Morera's theorem.
\end{proof}
\begin{bem}
    Note that by the results in  \cite{Biswas2019} $L_A$ generates a special hypergroup structure on $(0,\infty)$ called Ch\'ebli-Trim\`eche hypergroup. For a detailed introduction to the topic of hypergroups see \cite{bloom2011harmonic}.
    \end{bem}

\section{Convolution}\label{s3}
Assume $X$ to be a non-compact simply connected harmonic manifold of rank one. 
 Let $f$ be radial around $\sigma\in X$ with $f=u\circ d_{\sigma}$ for some function $u:\R\to \R$. For $x\in X$ define the $x$-translate of $f$ by 
\begin{align*}
\tau_x f :=u\circ d_x.
\end{align*}
Observe that 

\begin{enumerate}
\item $\tau_x f(y)=u\circ d(x,y)=\tau_{y}f(x)$.
\item If $f\in L^p (X)$ $1\leq p\leq \infty $ is radial  we have $\lVert f \rVert_p=\lVert \tau_x f \rVert_p$.

\end{enumerate}

\begin{defi}
Let $g:X\to \C$ be a radial function  around $\sigma\in X$, with $g=u\circ d_\sigma$ for some even function $u:\R\to \C$.
 Then the convolution of a function $f:X\to \C$ with compact support with $g$
   is given by $$ (f*g)(x):=\int_{X} f(y)\cdot\tau_x g(y)\,dy.$$
\end{defi}
\begin{lemma}
    Let $f\in L^p(X)$ $1\leq p\leq \infty$ and $g \in L^1(X)$ radial hence $g=u\circ d_{\sigma}$ for some even $u:\R\to \C$. Then $ f*g \in L^p(X)$.
\end{lemma}
\begin{proof}
First suppose $f\in L^1(X)$ then:
\begin{align*}
\lVert f*g\rVert_1&\leq \int_X\int_X\lvert f(y)\rvert\lvert(\tau_xg)(y)\rvert \,dy\,dx\\
&=\int_X\lvert f(y)\rvert\Bigl(\int_{0}^{\infty}\lvert u(r)\rvert A(r) \,dr\Bigr) \,dy\\
&=\lVert f\rVert_1\lVert g\rVert_1< \infty. 
\end{align*}
Now replace  $f\in L^1(X)$ by $f\in L^{\infty}(X)$ in the calculation above we obtain:
\begin{align*}
\lVert f*g\rVert_{\infty}\leq \lVert f\rVert_{\infty}\cdot \lVert g\rVert_1.
\end{align*}
Hence, by the Rize-Thorin theorem, we get for 

 for all $ p\in[1,\infty]$ and $f\in L^p(X),$
\begin{align}\label{eq:conalg}
\lVert f*g\rVert_p\leq \lVert f\rVert_{p}\cdot\lVert g\rVert_1.
\end{align}
\end{proof}

\begin{bem}\label{Bem2.4.5}
\begin{enumerate}
    \item Let $f:X\to\C$, $g:X\to\C$ be smooth with compact support and $g$ radial around $\sigma\in X$. Then by \cite{Biswas2019} we have:
$$\widetilde{f*g}^{\sigma}=\tilde{f}^\sigma\cdot \widehat{g}.$$
Furthermore, the statement holds almost everywhere if the Fourier transforms exist almost everywhere. 

\item By the Plancherel theorem for the Fourier transform we get
$$ \lVert f*g\rVert_2 =\lVert \tilde{f}^\sigma\cdot \hat{g}\rVert_2.$$
for suitable functions $f,g$. 
\item 
If $f,g\in L^1(X)$ are  radial around $\sigma\in X$ then  $ f*g=g*f\in L^1(X)$ and the convolution  is radial around $\sigma$.
\item  The radial  $L^1$-functions around a point form a commutative Banach algebra under convolution \cite[Theorem 7.2]{Biswas2019}. Note that this also can be shown without the assumption of purely exponential volume growth as done in \cite{PS15}.
\end{enumerate}
\end{bem}

The following Lemma is a Young-type inequality for the convolution on harmonic manifolds which will be needed in the later discussion.  Note that this also holds without the assumption of rank one. The proof can be found in \cite{brammen2023dynamics}.
\begin{lemma}\label{young}
Let $p,q,r>0$ and satisfy: $1+\frac{1}{r}=\frac{1}{p}+\frac{1}{q}$. Then for $f\in L^p (X)$ and $g\in L^q(X)$, $g=u\circ d_\sigma$ radial around $\sigma\in X$ we get:
\begin{align*}
&f*g\in L^r(X)\\
&\text{ and }\\
&\Vert f*g\Vert_{r}\leq \Vert f\Vert_p \cdot \Vert g \Vert_q.
\end{align*}

\end{lemma}

\section{Bounds on Eigenfunctions}\label{s4}
As in previous sections $\varphi_{\lambda,\sigma}$ denotes the radial eigenfunction of $\Delta$ with eigenvalue $-(\lambda^2+\rho^2)$ for $\lambda\in \C$ normalised at and radial around $\sigma\in X$. 
Since the radial eigenfunction of the Laplacian on $X$ coincides with the eigenfunctions of the hypergoupe associated with $A(r)$ we obtain from  \cite[Prop. 6.1.1 and Prop. 6.1.4]{trimeche2018generalized}
the following lemma.

\begin{lemma}\label{lemma:bounds}
For all $x\in X$ and $\lambda\in \C$ we have:
\begin{enumerate}
\item$ \lvert \varphi_{\lambda,\sigma}(x)\rvert\leq\varphi_{i\operatorname{Im}(\lambda),\sigma}(x)\leq \varphi_{0,\sigma}(x)\cdot e^{\lvert \operatorname{Im}(\lambda)\rvert d(\sigma,x)}$
\item $\lvert \operatorname{Im}(\lambda)\rvert \leq\rho \Rightarrow e^{(\lvert \operatorname{Im}(\lambda)\rvert -\rho)d(\sigma,x)}\leq \varphi_{\operatorname{Im}(\lambda),\sigma}(x)\leq 1$
\item $\lvert \operatorname{Im}(\lambda)\rvert \geq\rho \Rightarrow 1\leq \varphi_{i\operatorname{Im}(\lambda),\sigma}(x)\leq e^{(\lvert \operatorname{Im}(\lambda)\rvert -\rho)d(\sigma,x)}$
\end{enumerate}
and 
\begin{align*}
\varphi_{i\operatorname{Im}(\lambda),\sigma}(x)\leq k(1+d(\sigma,x))e^{(\lvert \operatorname{Im}(\lambda)\rvert -\rho)d(\sigma,x)}
\end{align*}
 for  some positive constant $k>0$. 
\end{lemma}

\begin{lemma}
Let $\lambda\in\R$ and $\varphi_{\lambda,x}$ be the eigenfunction of the Laplacian radial around and normalised at $x$ with eigenvalue $-(\lambda^2+\rho^2)$. Then:
\begin{enumerate}
\item $\lvert \varphi_{\lambda,x}(y)\rvert\leq 1 $ for every $y\in X$. 
\item $\lvert 1-\varphi_{\lambda,x}(y)\rvert \leq t^2 \cdot (\lambda ^2+\rho^2)$ where $t=d(x,y)$. 
 
\end{enumerate}
\end{lemma}
\begin{proof}
(1) is immediate from Lemma \ref{lemma:bounds}.
(2): We have $\varphi_{\lambda,x}=\varphi_{\lambda}\circ d_x$. 
Since $\varphi^{\prime}_{\lambda}(0)=0$ and by equation (\ref{equ:radiallaplace}) we have: 
\begin{align*}
\varphi^{\prime}_{\lambda}=-\frac{\lambda^2+\rho^2}{A(t)}\int_0^t \varphi_{\lambda}(r)A(r)\,dr.
\end{align*}
Hence
\begin{align*}
1-\varphi_{\lambda}(t)&=\varphi_{\lambda}(0)-\varphi_{\lambda}(t)\\
&=-\int_0^t \varphi^{\prime}_{\lambda}(r)\,dr\\
&=\int_0^t \frac{\lambda^2+\rho^2}{A(r)}\int_0^r\varphi_{\lambda}(\tau)A(\tau)\,d\tau\,dr\\
\end{align*}
With the substitution $r=tv$ and $\tau=tuv$ we obtain:
\begin{align*}
\cdots&=(\lambda^2+\rho^2) \int_0^1\frac{t}{A(tv)}\int_0^1 tv\cdot\varphi_\lambda (tuv)A(tuv)\,du\,dv\\
&\leq (\lambda^2+\rho^2)\cdot t^2 \int_0^1\int_0^1 \frac{A(tuv)}{A(tv)}v\lvert \varphi_{\lambda}(tuv)\rvert \,du\,dv
\end{align*}
Since $A(t)$ is monotone increasing and $u,v\leq 1$, $\frac{A(tuv)}{A(tv)}\leq 1$. Furthermore $\lvert \varphi_{\lambda}(t)\rvert\leq 1$  hence we conclude the assertion. 
\end{proof}

\section{Radon Transform}\label{s5}
The Radon transform on harmonic manifolds of rank one introduced by Rouvier \cite{Rouvire2021} is a generalisation of the Abel transform for radial functions defined in \cite{PS15}. 
In the following, we are going to briefly introduce this Radon transform and state some properties relevant to the discussion.

For $\sigma\in X$ and $\xi\in\partial X$ let  $H_{\xi,x}^s=B^{-1}_{\xi,\sigma}(s)$ denote the horospheres and $N(x)=-\grad B_{\xi,\sigma}(x)$.

The map  
\begin{align*}
\Psi_{\xi,s}:H_{\xi,\sigma}^0\to H_{\xi,\sigma}^s\\
 x\mapsto \exp(-sN(x))
 \end{align*}
 is a diffeomorphism 
  and
 \begin{align*}
  \Psi_{\xi}:\R\times H_{\xi,\sigma}^0\to X\\
  \Psi_{\xi}(s,x)=\Psi_{\xi,s}(x)
  \end{align*}
is an orientation preserving diffeomorphisem \cite[Proposition 3.1]{PS15}.
Hence, for a measurable function $f:X\to \C$  we get : 
\begin{align*}
 \int_{H_{\xi,\sigma}^s}f(z)\,dH_{\xi,\sigma}^s(z)=e^{sh}\int_{H_{\xi,\sigma}^0}f(\Psi_{\xi,s}(z))\,dH_{\xi,\sigma}^0(\Psi_{\xi,s}(z)).
 \end{align*}
\begin{defi}
For $f\in L^1(X)$ define the Radon transform $\mathcal{R}_{\sigma}(f):[0,\infty)\times \partial X \to \R $ by:
$$\mathcal{R}_{\sigma}(f)(s,\xi):=e^{-\rho s}\int_{H_{\xi,\sigma}^s}f(z)\,dH_{\xi,\sigma}^s(z)$$
for all $s\in \R$ and $\xi\in \partial X$ for which the integral  exists. 
\end{defi}
\begin{bem}
    Note that our definition and the one in \cite{Rouvire2021} differ by the factor $e^{-\rho s}$. Furthermore, the signs in \cite{Rouvire2021} are reversed.  We choose this factor out of convenience see  Remark \ref{lemma1}.
\end{bem}
\begin{lemma}
For $f\in L^1(X)$ the Radon transform exists for almost every $(s,\xi)\in \R\times \partial X$.
\end{lemma}
\begin{proof}
Since $f\in L^1(X)$ we have that:
\begin{align*}
\infty>\int_X\lvert f(x)\rvert \,dx=\int_{\R} \int_{H_{\xi,\sigma}^s} \lvert f(z)\rvert \,dH_{\xi,\sigma}^s(z).
\end{align*}
Therefor the $\int_{H_{\xi,\sigma}^s} f(z)\,dH_{\xi,\sigma}^s(z) <\infty$ for almost every $(s,\xi)\in \R\times \partial X$. Hence, the Radon transform exists almost everywhere.
\end{proof}
\begin{bem}\label{lemma1}
\begin{enumerate}

\item Let $\mathcal{F}$ be the standard Fourier transform on $\R$. Then we have for $f\in L^1(X)$:
\begin{align*}
\tilde{f}^{\sigma}(\lambda,\xi)&=\int_{X}f(x)e^{-(i\lambda+\rho)B_{\xi,\sigma}(x)}\,dx\\
&=\int_{-\infty}^{\infty}\int_{H_{\xi,\sigma}^s}f(z)e^{-(i\lambda+\rho)s}\,d H_{\xi,\sigma}^s(z)\,ds\\
&=\int_{-\infty}^{\infty}e^{-i\lambda s}e^{-\rho s}\int_{H_{\xi,\sigma}^s}f(z)\,d H_{\xi,\sigma}^s(z)\,ds\\
&=\int_{-\infty}^{\infty}e^{-i\lambda s} \mathcal{R}_{\sigma}(f)(s,\xi)\,ds\\
&=\mathcal{F}(\mathcal{R}_{\sigma}(f)(s,\xi))(\lambda).
\end{align*}
Hence, the Radon transform of a radial function is independent of $\xi\in\partial X$. 

\item The definition of the Radon transform coincides on radial functions with the Abel transform defined in \cite{PS15}.
\item  With the inversion formula for the Fourier transform we obtain an inversion formula for the Radon transform. 
\item Since the respective Fourier transforms are isometries, the Radon transform exists for every $f\in  L^2(X)$  for almost every $s\in \R$ and $\xi\in \partial X$. 
 \item  Rouviere \cite{Rouvire2021} showed the for $f,g\in C_c^{\infty}(X)$, and $g$

 radial around $\sigma\in X$  we have with $r=d(\sigma,x)$:
\begin{align*}
&(i)~\mathcal{R}_{\sigma}M_{\sigma} f(r)=\int_{\partial X} \mathcal{R}_{\sigma}(f)(r,\xi)\,d\mu_{\sigma}(\xi),\\
&(ii)~\mathcal{R}_{\sigma}(f*g)=(\mathcal{R}_{\sigma}f)*(\mathcal{R}_{\sigma}g).
\end{align*}
The convolution in $(ii)$ on the left-hand side is the convolution on $X$ and on the right-hand side it is the one on the real line.

 \end{enumerate}
\end{bem}
\begin{lemma}
Let $f\in L^1(X)$ such that the Radon transform around some point $\sigma \in X$ vanishes almost everywhere. Then $f=0$ almost everywhere. 
\end{lemma}
\begin{proof}
We only need that the Fourier transform of $f$ exists almost everywhere on the real line since then by \ref{lemma1} if the Radon transform vanishes almost everywhere then so does the Fourier transform and then by the Plancherel theorem the $L^2$ norm of $f$ vanishes hence $f=0$ almost everywhere. But by Lemma \ref{lemma:holo2} the Fourier transform of $f$ exists almost everywhere on $\R\times \partial X$. 
\end{proof}

\section{heat kernel and integral transforms}\label{s6}
The heat kernel $h_t:X\times X\to\R$ $t>0$  is the fundamental solution of the heat equation 
\begin{align*}
u:X&\times\R_{\geq 0}\to\C\\
\frac{\partial}{\partial t} u&=\Delta u,\\
u(x,0)&=f(x)\in C^{\infty}_c(X).
\end{align*}
Furthermore, it can be considered as the integral kernel of the heat semi-group 
$e^{t\Delta}:L^p(X)\to L^p(X),\,1\leq p\leq \infty.$ See for instance \cite{davies1990heat}. Harmonicity is by \cite{szabo1990} equivalent to the fact that the heat kernel is radial i.e. $h_t(\sigma,\cdot)$ is a radial function for every $\sigma\in X$. Moreover $\tau_xh_t(\sigma,\cdot)=h_t(x,\cdot)$ since the heat kernel is unique. 
Now we want to calculate the Fourier transform of the heat kernel. 
 Let $x,y,\sigma\in X$. Consider a radial eigenfunction of the Laplacian  $\varphi_{\lambda,\sigma}$ with $\Delta \varphi_{\lambda,\sigma}=-(\lambda^2+\rho^2)\varphi_{\lambda,\sigma}$
 then 
 \begin{align*}
 \frac{\partial}{\partial t} e^{-t(\lambda^2+\rho^2)}\varphi_{\lambda,\sigma}=\Delta e^{-t(\lambda^2+\rho^2)}\varphi_{\lambda,\sigma},
 \end{align*}
 hence $u(t,x):=e^{-t(\lambda^2+\rho^2)}\varphi_{\lambda,\sigma}(x)$ is a solution of the heat equation with $u(0,x)=\varphi_{\lambda,\sigma}(x)$.
 Therefore we have, since $h_t$ is the fundamental solution of the heat equation:
 \begin{align}\label{eq:heatfourier}
 \hat{h}_t^{\sigma}(\lambda)&=\int_X \varphi_{\lambda,\sigma}(x)h_t(\sigma,x)\,dx\\
 &=u(t,\sigma)\nonumber\\
 &=e^{-t(\lambda^2+\rho^2)}\varphi_{\lambda,\sigma}(\sigma)\nonumber\\
 &=e^{-t(\lambda^2+\rho^2)}\nonumber.
 \end{align}
\begin{bem}
\begin{enumerate}
    
\item  By using the Remark \ref{lemma1}  and the formula for the Fourier transform of the heat kernel above we get:
$$\mathcal{R}_x(h_t(x,y))(s)=e^{-\rho^2}\frac{1}{\sqrt{4\pi t}}e^{-s^2/4t}.$$
This coincides with Theorem 5.6 in \cite{PS15}.
Note that the result in \cite{PS15} was obtained without the rank one assumption. 
\item Setting $\lambda=i\rho$ in equation (\ref{eq:heatfourier})  we obtain $\lVert h_t(\sigma,\cdot)\rVert_1=1$ for every $\sigma\in X$. Hence $(X,g)$ is stochastically complete. 
\end{enumerate}
\end{bem}

From the above and  Remark \ref{Bem2.4.5} we have that: 
\begin{lemma}\label{heat2}
Let $f\in L^1(X)$ 

and $h_t$ the heat kernel at $\sigma$ then:
$$\widetilde{f*h_t} ^\sigma (\lambda,\xi)= e^{-t(\lambda^2 +\rho^2)}\Tilde{f}^\sigma(\lambda,\xi) $$
almost everywhere. Moreover, using the Young inequality $f*h_t\in L^2(X). $
\end{lemma}

The following theorem is essentially contained in \cite{Heisenberg} but for the reader's convenience we provide a proof. 
\begin{satz}\label{ver2}
Let $(X,g)$ be a non-compact simply connected harmonic manifold of rank one, with mean curvature of the horospheres $2\rho$ and $h_t^{\sigma}(x)=h_t(\sigma,x)$ the heat kernel centred at a point $\sigma\in X$. Then
$$\Vert h_t^{\sigma}\Vert_2 \asymp \begin{cases}  t^{-(\alpha+1)/2} & \text{if} ~t <<1 \\ 
e^{t\rho^2}t^{-(\mu+1)/2}& \text{if} ~t> >1 ,
\end{cases}$$
with $\alpha=\frac{n-2}{2}$ and $\mu=\frac{1}{2}$.
\end{satz}
The proof follows closely \cite{Heisenberg} with a simplification in the second part.
\begin{proof}
Recall that
the Fourier transform is an isometry with respect to the $L^2$ norm. Hence, for $r=d(\sigma,x)$ we have 
\begin{align*}
\Vert h_t^{\sigma} \Vert_2^2= \int_0^{\infty}\vert h^{\sigma}_t (r)\vert^2 A(r)dr=C_0\int_{0}^{\infty}\vert \widehat{h}_t^{\sigma} (\lambda)\vert^2\vert \mathbf{c}(\lambda)\vert^{-2}\,d\lambda.
\end{align*}
Using the Fourier transform of the heat kernel yields
\begin{align*}
    \Vert h_t^{\sigma}\Vert_2 &=\left(C_0 \int_{0}^{\infty}e^{-2t(\lambda^2 + \rho^2)}\vert \mathbf{c}(\lambda)\vert^{-2}\,d\lambda\right)^\frac{1}{2}\\
    &=e^{-t\rho^2}\left(C_0\int_{0}^{\infty}e^{-2t\lambda^2}\vert \mathbf{c}(\lambda)\vert^{-2}d\lambda\right)^{\frac{1}{2}}.
\end{align*}
With the estimate of the $\mathbf{c}$-function 
\begin{align}\label{cfunction}
\vert \mathbf{c}(\lambda)\vert^{-2} \asymp 
\begin{cases} \lvert \lambda\rvert^2 & \text{if} ~\lvert \lambda\rvert  \leq C_{\mathbf{A}} \\ 
\lvert \lambda\rvert ^{2\alpha +1 }& \text{if} ~\lvert \lambda\rvert> C_{\mathbf{A}}.
\end{cases}
\end{align}
from \cite{Bloom1995TheHM}.

and by dropping all constants there is  some constant $K>0$ such that
\begin{align*}
\Vert h_t^{\sigma}\Vert_2 \asymp e^{t\rho^2}\cdot\left(\int_{0}^{K}e^{-2t\lambda^2}\lambda^{2\mu+1}\,d\lambda +\int_{K}^{\infty}e^{-2t\lambda^2}\lambda^{2\alpha+1}\,d\lambda\right)^{\frac{1}{2}}.
\end{align*}
Using the coordinate transform $s=2t\lambda^2$ and by again dropping all constants yields:
\begin{align*}
\Vert h_t^{\sigma}\Vert_2 \asymp e^{t\rho^2}\cdot\left(t^{-(\mu+1)}\int_{0}^{2tK^2}e^{-s}s^{\mu}\,ds + t^{-(\alpha+1)}\int_{2tK^2}^{\infty}e^{-s}s^{\alpha}\,ds\right)^{\frac{1}{2}}.
\end{align*}
The next step is to investigate the asymptotic behaviour of those terms.\\
First, consider the case $t\in[0,1]$. Then
   $t^{-(\mu+1)}\int_{0}^{2tK^2}e^{-s}s^{\mu}\,ds$ is bounded and  
   \begin{align*}
   \int_{2tK^2}^{\infty}e^{-s}s^{\alpha}ds\leq\int_{0}^{\infty}e^{-s}s^{\alpha}\,ds
   \end{align*}
   is bounded away from infinity.
    Hence
    \begin{align*}
    \Vert h_t^{\sigma}\Vert_2 \asymp t^{-(\alpha+1)/2}\quad \text{ for } t<<1.
    \end{align*}
    Assume $t\geq 1$, then:  
 \begin{align*}
 t^{-(\mu +1)}\int_{0}^{2tK^2}e^{-s}s^{\mu}ds\asymp t^{-(\mu +1 )}.
 \end{align*}
Furthermore
 \begin{align*}
  t^{-(\alpha+1)}\int_{2tK^2}^{\infty}e^{-s}s^{\alpha}\,ds\leq  \int_{2K^2}^{\infty}e^{-s}s^{\alpha}\,ds,
 \end{align*}
 where the right-hand side is finite only depending on $\alpha$ and $K$.
Therefore:
\begin{align*}
\Vert h_t^{\sigma}\Vert_2\asymp e^{t\rho^2}t^{-(\mu+1)/2} \quad\text{ for } t\geq1.
\end{align*}
 \end{proof}

The Lemma \ref{heat2} tells us that the heat kernel is a smoothing tool. This observation leads to the following: 
\begin{prop}
Let  $f\in L^{1}(X)$ then for almost every $x\in X$ 

$$f(x)=C_{0}\int_{0}^{\infty}\int_{\partial X}\widetilde{f}^{\sigma}(\lambda,\xi)e^{(i\lambda-\rho)B_{\xi,\sigma}(x)}\,d\mu_{\sigma}(\xi)\lvert \mathbf{c}(\lambda)\rvert ^{-2}\,d\lambda.$$
\end{prop}
\begin{proof}
Let $h_{t}$ be the heat kernel around $\sigma\in X$ and $\phi\in C^{\infty}_{c}(X)$ then $f*h_{t}\in L^{2}(X)$ by the Young inequality with $r=q=2$ and $p=1$. Hence by the Placharel theorem 
$$ \langle f*h_{t},\phi\rangle=\langle F_{t},\phi\rangle,$$
where 
$$F_{t}(x)=\int_{0}^{\infty}\int_{\partial X}\widetilde{f}^{\sigma}(\lambda,\xi)e^{-t(\lambda^{2}+\rho^{2})}e^{(i\lambda-\rho)B_{\xi,\sigma}(x)}\,d\mu_{\sigma}(\xi)\lvert \mathbf{c}(\lambda)\rvert^{-2}\,d\lambda.$$ 
we have that $\lvert F_{t}\rvert=\lvert f*h_{t}\rvert <f$ since $f*h_{t}$ converges for $t\to 0$ uniform to $f$. Hence the claim follows by the dominant convergence theorem. 
\end{proof}

\section{Heisenberg uncertainty Principle}\label{s7}
A classical result going back to the  1920s about the Fourier transform for $L^2$-functions on the real line is the uncertainty principle, stating that one can not have a precise localisation and a high spectral resolution at the same time. In the following, we are going to deduce such a relation between a function and its Fourier transform. The main tool to obtain this is the estimate on the $L^2$ norm of the heat kernel from Theorem \ref{ver2}. The proofs in this section follow the one in \cite{Heisenberg}  closely with some amendments made for the non-radial case.

\begin{lemma}\label{heis1}
Let $\sigma\in X$, $\gamma\in D_p$, $\alpha=\frac{n-2}{2}$ and $a>0$ such that $\gamma a<\alpha +1 $ then there exist a constant $C>0$ such that 
$$\Vert f*h_t(\sigma,\cdot)\Vert_2\leq Ct^{-a/2}\Vert d_\sigma^{\gamma a}f\Vert_2, \quad\forall f\in 
L^2(X).$$
\end{lemma}
\begin{proof}

First, we are going to prove the assertion for radial functions and then adapt the critical steps in the proof to non-radial functions. 
For this interpret the heat kernel on $X$  $h_t$ again as a function on the positive half line.
Further assume $f=u \circ d_{\sigma}$ for $u\in L^2(\R _+,A(x)\,dx)$. 
 Let $r>0$, $u_r:=u\cdot \chi_{[0,r]},~ u^r:=u-u_r$ then  
\begin{align*}
\vert u^{r}(x)\vert\leq r^{-\gamma a}\vert x^{\gamma a}u(x)\vert\quad \forall x\in \R _+.
\end{align*}
This is because in the case $x\leq r$  the inequality above is trivial and in the case $x>r$ we have $(\frac{x}{r})^{\gamma a}\geq 1$. 
Using Lemma \ref{heat2} and the Plancherel theorem we get on  one hand:
\begin{align*}
    \Vert u^r *h_t\Vert_2 &\leq \Vert e^{-t(\lambda^2+\rho^2)}\widehat{u}^r(\lambda)\Vert_2\\
    &\leq\Vert \widehat{u}^{r}\Vert_2=\Vert u^r\Vert_2\\
    &\leq r^{-\gamma a}\Vert x^{\gamma a}u\Vert_2.
\end{align*}
On the other hand by the Young inequality:
\begin{align*}
\Vert u_r*h_t\Vert_2\leq \Vert h_t\Vert_2 \Vert u_r\Vert_1.
\end{align*}
With the H\"older inequality we obtain:
\begin{align*}
\Vert u_r \Vert_1 &=\Vert \chi_{[0,r]}\frac{x^{\gamma a}}{x^{\gamma a}}u\Vert_1\\
&=\Vert\chi_{[0,r]}\cdot x^{-\gamma a}\cdot x^{\gamma a}u\Vert_1\\
&\leq \Vert \chi_{[0,r]}\cdot x^{-\gamma a}\Vert_2 \cdot \Vert x^{\gamma a}u\Vert_2\\
&=\Bigl(\int_0^{r}x^{-2\gamma a}A(x)\,dx\Bigr)^{1/2}\Vert x^{\gamma a}u\Vert_2.
\end{align*}
Hence, we conclude:
\begin{align*}
\Vert u_r*h_t\Vert_2\leq \Vert h_t \Vert_2 \cdot \Bigl(\int_0^{r}x^{-2\gamma a}A(x)\,dx\Bigr)^{1/2}\Vert x^{\gamma a}u\Vert_2.
\end{align*}
Since $X$ has purely exponential volume growth and $A(r)$ is the determinant of the Jacobi tensor with initial conditions $A_v(0)=0$ and $A'(0)=\operatorname{id}$ we have $A(x)\sim e^{2\rho x}$ as $x\to\infty$ and $A(x)\sim x^{2\alpha +1}$ as $x\to 0$.
So in the case $r<1$ we get with the assumption on $\gamma a$:
\begin{align*}
    \Bigl(\int_{0}^{r}x^{-2\gamma a}\cdot A(x)\,dx\Bigr)^{1/2}&\leq c_{1,1}\cdot \Bigl(\int_{0}^{r}x^{-2\gamma a}\cdot x^{2\alpha +1}\,dx\Bigr)^{1/2}\\
    &\leq c_{1,2}\cdot r^{\alpha +1}\cdot r^{-\gamma a}
\end{align*}
for some constants $c_{1,1},c_{1,2}>0$.
And in the case $r\geq1$ we get:
\begin{align*}
    \Bigl(\int_{0}^{r}x^{-2\gamma a}\cdot A(x)\,dx\Bigr)^{1/2}&\leq c_{2,1}\cdot \Bigl(\int_{0}^{r}x^{-2\gamma a}\cdot e^{2\rho x}\,dx\Bigr)^{1/2}\\
    &\leq c_{2,2}\cdot r^{1/2}\cdot e^{\rho r}\cdot r^{-\gamma a}
\end{align*}
for some constants $c_{2,1},c_{2,2}>0$.
Hence, with $c=\max\{c_{1,2},c_{2,2}\}$, we define 
\begin{align*} V(r)=\begin{cases}
c \cdot r^{\alpha +1}&\text{if}~r<1\\
c \cdot r^{1/2}e^{pr}&\text{if}~r\geq 1.
\end{cases}
\end{align*}
By the triangle inequality, we get:
\begin{align*}
\Vert u*h_t\Vert_2&\leq \Vert u_r *h_t\Vert_2+\Vert u^r *h_t\Vert_2\\
&\leq C r^{-\gamma a}(1+\Vert h_t\Vert_2 V(r)\Vert x^{\gamma a}u\Vert_2),
\end{align*}
for $C>0$ only dependant on the variables $\gamma, a$.
If we now choose $r=t^{1/(2\gamma)}$ and use the bounds on the heat kernel from \ref{ver2} to get the desired inequality for radial functions.

It remains to consider the non-radial case. For $f\in L^2(X)$ non-radial consider the geodesic ball $B_r=\{y\in X\vert~d_\sigma (y)\leq r\}$, $f_r:=f\cdot \chi_{B_r}$ and $f^r:=f-f_r$. Then we obtain via the same line of argument as before: \begin{align*}
 \Vert f^r *h_t(\sigma,\cdot)\Vert_2 \leq  r^{-\gamma a}\Vert d_\sigma ^{\gamma a}f\Vert_2
\end{align*}
Since in Lemma \ref{young} we showed that the Young inequality also holds for the convolution of a non-radial with a radial function, we get:
\begin{align*}
\Vert f_r*h_t(\sigma,\cdot)\Vert_2\leq \Vert h_t(\sigma,\cdot) \Vert_2 \cdot  \Bigl(\int_0^{r}x^{-2\gamma a}A(x)\,dx\Bigr)^{1/2}\Vert d_\sigma^{\gamma a}f\Vert_2.
\end{align*}
 The same argument as in the radial case concludes the proof. 
\end{proof}

\begin{satz}\label{Vermutung3}  Let $(X,g)$ be a non-compact simply connected harmonic manifold of rank one, with mean curvature of the horosphere $2\rho$. 
And $D_\rho:=(\frac{1}{2},1]$ for $0<\rho<1$ or $D_\rho:=[\frac{1}{2},1]$ for $\rho\geq1$. Assume $a,b>0$, $\gamma\in D_\rho$ and $\sigma\in X$. Then there exist a constant $C>0$ such that:
$$
C\Vert f\Vert_2 \leq \Vert d_{\sigma}^{\gamma a}f\Vert_2 ^{b/(a+b)}\Vert (\lambda^2+\rho^2)^{b/2}\tilde{f}^{\sigma}\Vert_2 ^{a/(a+b)}
$$ for all $f \in 
 L^2(X)$, where $d_{\sigma}$ is the distance function  centred at $\sigma\in X$ and   $n\geq 6$ the dimension of $X$. 
\end{satz}

\begin{proof}
The proof is going to be split into three steps.\\
\underline{Step 1:} Assume $\gamma\in D_p$, $a$ such that $\gamma a<\alpha+1$ and $b\leq 2$. Then by the triangle inequality, Lemma \ref{heis1} and the Plancherel Theorem  we get: 
\begin{align*}
    \Vert f\Vert_2&\leq \Vert f*h_t(\sigma,\cdot)\Vert_2+\Vert f-f*h_t(\sigma,\cdot)\Vert_2\\
    &\leq C_{0}t^{-a/2}\Vert d_{\sigma}^{\gamma a}f\Vert_2+\Vert(1-e^{-t(\lambda^2+\rho^2)})\tilde{f}^{\sigma}(\lambda)\Vert_2\\
    &= C_{0}t^{-a/2}\Vert d_{\sigma}^{\gamma a}f\Vert_2\\&+
    \Vert(1-e^{-t(\lambda^2+\rho^2)})(t(\lambda^2+\rho^2))^{-b/2}(t(\lambda^2+\rho^2))^{b/2}\tilde{f}^{\sigma}(\lambda)\Vert_2.
\end{align*}
For $b\leq2$ the function $(1-e^{s})s^{-b/2}$ is bounded for $s\geq0$. Hence, we obtain:
\begin{align*}
\Vert f\Vert_2\leq C(t^{-a/2}\Vert d_{\sigma}^{\gamma a}f\Vert_2+t^{b/2}\Vert(\lambda^2+\rho^2)^{b/2}\tilde{f}^{\sigma}(\lambda)\Vert_2).
\end{align*}
Optimising in $t$ we get Theorem \ref{Vermutung3}  for $\gamma a<\alpha+1$ and $b\leq 2$.\\

\underline{Step 2:} Let $b>2$ and $b^{\prime}\leq 2$. Then for $u>0$ 
\begin{align*}
u^{b^{\prime}}\leq1+u^{b}.
\end{align*}
For $\epsilon>0$ set $u=((\lambda^2+\rho^2)/\epsilon)^{1/2}$. This yields:
\begin{align*}
\lVert (\lambda^2+\rho^2)^{b^{\prime}/2}\tilde{f}^{\sigma}\rVert_2\leq \epsilon^{b^{\prime}/2}\lVert f\rVert_2+\epsilon^{(b^{\prime}-b)/2}\lVert (\lambda^2+\rho^2)^{b/2}\tilde{f}^{\sigma}\rVert^{b^{\prime}/b}.
\end{align*}
 Optimising in $\epsilon$ we obtain:
\begin{align*}
\Vert (\lambda^2+p^2)^{b^{\prime}/2}\tilde{f}^{\sigma}\Vert_2\leq \Vert f\Vert_2^{(1-b^{\prime}/b)}\Vert (\lambda^2+\rho^2)^{b/2}\tilde{f}^{\sigma}\Vert_2^{b^{\prime}/b}.
\end{align*}
Applying  Theorem \ref{Vermutung3} for $b^{\prime}$ gives us Theorem \ref{Vermutung3} for $b$.\\
\underline{Step 3:} Assume $\gamma a\geq \alpha +1$ and choose $a^\prime$ such that $\gamma a^\prime<\alpha+1$. Then again we have for $\epsilon>0$:
\begin{align*}
\frac{x^{\gamma a^{\prime}}}{\epsilon^{\gamma a^{\prime}}}\leq 1+\frac{x^{\gamma a}}{\epsilon^{\gamma a}}
\end{align*}
and get:
\begin{align*}
\lVert d_\sigma^{\gamma a}f\lVert \leq \epsilon^{\gamma-a^\prime}\lVert f \rVert_2 \epsilon^{\gamma(a^\prime-a)}+\lVert d_\sigma^{\gamma a}f\rVert_2.
\end{align*}
Optimising in $\epsilon$  and Theorem \ref{Vermutung3}  for $a^\prime$ yields the result.
\end{proof}
\begin{bem}
  Note that due to the construction the constant $C$ heavily depends on $a$ and $b$.  
\end{bem}

\section{Morgan Theorem}\label{s8}

Consider the Fourier transform $\mathcal{F}^n$ on $\R^n$. Then we have the following uncertainty principle. 
\begin{satz}[\cite{bagchiray1998}]\label{thm:morgr}
Let $f:\R^n\to\C$ be a measurable function and assume that for all $x,\xi \in \R^n$ 
\begin{enumerate}
\item$ \lvert f(x)\rvert\leq Ce^{-\alpha\lvert x\rvert^{p}}$
\item $\lvert \mathcal{F}^n(f)(\xi)\rvert\leq Ce^{-\beta\lvert\xi\rvert^{q}},$
\end{enumerate}
where $C,\alpha,\beta>0$, $1<p,q<\infty$ and $\frac{1}{p}+\frac{1}{q}=1$
if $(\alpha p)^p (\beta q)^q >1$. Then $f=0$ almost everywhere.
\end{satz}
The authors in \cite{uncertaintysym} generalised this to non-compact symmetric spaces: 
Let $G$ be a connected non-compact semi-simple Lie Group with finite centre, $K\subset G$ a maximal compact subgroup and $X'=G/K$ the associated Riemann symmetric space of non-compact type. Denote by $o=eK$ the origin in $X$ and by $d:X'\times X'\to\R$ the Riemannian distance. Furthermore, let $KAN$ be the Iwasawa decomposition and $\mathcal{A}$ the Lie algebra of $A$ and let its dual be denoted by $\mathcal{A}^*$. Now let $M$ be the centraliser of $A$ in $K$ and $B=K/M$ the boundary of $X'$. 
\begin{satz}[\cite{uncertaintysym}]
Let $f:X'\to \C$ be a measurable function and assume that for all $\lambda\in \mathcal{A}^*,b\in B\text{ and } x\in X'$ we have
\begin{enumerate}
\item$ \lvert f(x)\rvert\leq Ce^{-\alpha d(o,x)^{p}}$,
\item $\lvert \tilde{f}(\lambda,b)\rvert\leq Ce^{-\beta\lvert\lambda\rvert^{q}}$,
\end{enumerate}
where $C,\alpha,\beta>0$, $1<p,q<\infty$ and $\frac{1}{p}+\frac{1}{q}=1$.
If $(\alpha p)^p (\beta q)^q >1$. Then $f=0$ almost everywhere.
\end{satz}
The following classical result from \cite{GeoHoro} allow us to compare the distance on horospheres with the ambient distance and is the main reason we can only consider harmonic manifolds with pinched negative curvature. 
\begin{lemma}[{\cite[Theorem 4.6]{GeoHoro}}]\label{lemma:GeoHoro}
Let $(M,g)$ be a Hadarmard manifold with pinched negative curvature, $d$ the Riemannian distance and $h$  the induced distance on a horosphere. Let $p$ and $q$ be two points on the horosphere then there exist constants $C_1,C_2>0$ such that:
\begin{align*}
C_1 d(p,q)\leq h(p,q)\leq C_2 d(p,q)
\end{align*}
for all $p,q\in M$. 
\end{lemma}

\begin{prop}\label{prop:boundradon}
Let $f\in L^1(X)\cap L^2(X)$  satisfy the conditions of  Theorem \ref{morgan} for some point $\sigma\in X$. Then for each $\alpha_1$ with $0<\alpha_1<\alpha$ there exists a constant $E>0$ such that:
$$ \lvert \mathcal{R}_\sigma(f)(s,\xi)\rvert\leq E\cdot e^{-\alpha_1 \lvert s\rvert^p}\quad \forall (s,\xi)\in \R\times \partial X.$$
 If $f\in L^1(X)$  satisfies the conditions of Theorem \ref{thm:schroe} then for each $\alpha_1$ with $0<\alpha_1<\alpha$ there exists a constant $E>0$ such that:
$$ \lvert \mathcal{R}_\sigma(f)(s,\xi)\rvert\leq E\cdot e^{-\alpha_1 \lvert s\rvert^2}\quad \forall (s,\xi)\in \R\times \partial X.$$

\end{prop}
\begin{proof}
By the conditions on $f$ and the monotonicity of the integral we have:
\begin{align}
\lvert \mathcal{R}_\sigma(f)(s,\xi)\rvert&=\left\lvert e^{-\rho s}\int_{H_{\xi,\sigma}^s}f(x) \,dH_{\xi,\sigma}^s(x)\right\rvert\nonumber\\
&\leq C e^{-\rho s}\int_{H_{\xi,\sigma}^s} e^{-\alpha d_\sigma^{p}(x)}\,dH_{\xi,\sigma}^s(x).\label{4.43}
\end{align}
Let $\varphi_{0,\sigma}$ be the radial eigenfunction of the Laplacian with eigenvalue $\rho^2$. Then (\ref{4.43}) becomes
\begin{align}\label{1}
C e^{-\rho s}\int_{H_{\xi,\sigma}^s}\varphi_{0,\sigma}(x) e^{-\alpha d_\sigma^{p}(x)}\varphi_{0,\sigma}^{-1}\,dH_{\xi,\sigma}^s(x).
\end{align}
Since these eigenfunctions correspond to the eigenfunctions of the Ch\'ebli-Trim\`eche hypergroups on $(\R^{+},*)$ associated to $A(r)$, we get from Lemma \ref{lemma:bounds} 
\begin{align*}
e^{\rho d_{\sigma}(x)}\leq \varphi_{0,\sigma}(x)\Leftrightarrow \varphi_{0,\sigma}^{-1}(x)\leq e^{-\rho d_\sigma (x)}.
\end{align*}
Since the Busemann function is Lipschitz with Lipschitz-constant $1$, we get
\begin{equation}\label{2}
\rvert s\lvert=\lvert B_{\xi,\sigma}(x)\rvert\leq d_{\sigma}(x).
\end{equation}
Let $\epsilon=\alpha-\alpha_1$. Then we get:
\begin{align*}
&\int_{H_{\xi,\sigma}^s}\varphi_{0,\sigma}(x) e^{-\alpha d_\sigma^{p}(x)}\varphi_{0,\sigma}^{-1}(x)\,dH_{\xi,\sigma}^s(x)\\
&=\int_{H_{\xi,\sigma}^s}\varphi_{0,\sigma}(x) e^{-\alpha_1 d_\sigma^{p}(x)}e^{-\epsilon d_{\sigma}^{p}(x)}\varphi_{0,\sigma}^{-1}(x)\,dH_{\xi,\sigma}^s(x).
\end{align*}
With (\ref{1}) and (\ref{2}) we obtain:
\begin{align*}
\dots &\leq \int_{H_{\xi,\sigma}^s} \varphi_{0,\sigma}(x)e^{-\alpha_1 d_{\sigma}^{p}(x)}e^{-\epsilon d_{\sigma}^{p}(x)}e^{-\rho d_\sigma (x)}\,dH_{\xi,\sigma}^s(x)\\
&\leq e^{-\alpha_1 \rvert s\lvert^p}\int_{H_{\xi,\sigma}^s}\varphi_{0,\sigma}(x)e^{-\epsilon d_{\sigma}^{p}(x)} e^{\rho d_{\sigma}(x)}(1+d_{\sigma}(x))(1+d_{\sigma}(x))^{-1} \,dH_{\xi,\sigma}^s(x).
\end{align*}
Since the function $h(y)=e^{-\epsilon y^p+\rho y}(1+y)$ is bounded on $\R$, its supremum $D<\infty$  exists and we get
\begin{align*}
\dots \leq e^{-\alpha_1 \lvert s\rvert ^{p}} D \int_{H_{\xi,\sigma}^s}\varphi_{0,\sigma}(x)(1+d_{\sigma}(x))^{-1} \,dH_{\xi,\sigma}^s(x).
\end{align*}
From 
Lemma \ref{lemma:bounds}
 we have the following: There exists a constant $k>0$ such that 
 \begin{align*}
 \varphi_{0,\sigma}(x)\leq k(1+d_{\sigma}(x))e^{-\rho d_{\sigma}(x)}~\forall x\in X.
 \end{align*}
This 
yields:
\begin{align*}
\dots &\leq e^{-\alpha_1\lvert s\rvert^p} D k\int_{H_{\xi,\sigma}^s} e^{-\rho d_{\sigma}(x)}\, dH_{\xi,\sigma}^s(x).
\end{align*}
Now let $h_s$ denote the distance induced on $H_{\xi,\sigma}^s$ by the ambient space and $\sigma(s)$ the intersection point of the geodesic perpendicular to $H_{\xi,\sigma}^0$ in $\sigma$ and $H_{\xi,\sigma}^s$. By Lemma \ref{lemma:GeoHoro} we have $0<c\leq 1$ such that: 
\begin{align*}
c\cdot h_s(\sigma(s),x)\leq d(\sigma,x)+s.
\end{align*}
Hence, we obtain: 
\begin{align*}
\dots &\leq e^{-\alpha_1\lvert s\rvert^p} D k\int_{H_{\xi,\sigma}^s} e^{-\rho d_{\sigma}(x)}\, dH_{\xi,\sigma}^s(x)\\
 &\leq e^{-\alpha_1\lvert s\rvert^p} D ke^{\rho s }\int_{H_{\xi,\sigma}^s} e^{-\rho \cdot c\cdot h_s(\sigma(s),(x))}\, dH_{\xi,\sigma}^s(x).
\end{align*}
Since the horospheres have polynomial volume growth bounded by a polynomial without constant factor, where the exponent only depends on the mean curvature of the horosphere which is constant and the decay constant of the stable manifold \cite[Section 19]{knieper22016}, the above integral converges and is bounded by a constant independent of $s$ yielding the assertion. 
\end{proof}

\begin{satz}\label{morgan}
Let $(X,g)$ be a non-compact simply connected harmonic manifold with pinched negative curvature. 
Let $f\in L^1(X)\cap L^2 (X)$ such that for some $\sigma\in X$, all $x\in X $, $ \lambda \in \R,\text { and } \xi\in\partial X$ we have
\begin{enumerate}
\item $\lvert f(x) \rvert \leq C e^{-\alpha (d_{\sigma}(x))^p}$,
\item $\lvert \tilde{f}^{\sigma}(\lambda,\xi)\rvert\leq C e^{-\beta \lvert \lambda \rvert^q}$,
\end{enumerate}
where $C,\alpha, \beta$ are positive constants and
$1<p,q<\infty,\text{ with } \frac{1}{p}+\frac{1}{q}=1$.
If $(\alpha p)^{1/p}(\beta q)^{1/q}>1$ then $f=0$ almost everywhere.
\end{satz}

\begin{proof}

By Lemma \ref{prop:boundradon} for 
$f\in L^1(X)\cap L^2(X)$,  which satisfies the conditions of  Theorem \ref{morgan}, we get 
\begin{align*}
\mathcal{F}(\mathcal{R_{\sigma}}(f))=\tilde{f}^\sigma.
\end{align*}
 Now choose $0<\alpha_1<\alpha$ such hat
\begin{align*}
(\alpha_1 p )^{1/p}(\beta q)^{1/q}>1.
\end{align*}
Then, by the Morgan theorem on the real line, $\mathcal{R}(s,\xi)$ is zero almost everywhere.  
Hence by Remark \ref{lemma1}, $\tilde{f}^\sigma$ is zero almost everywhere. By the Plancherel Theorem on $X$ this implies the vanishing of the $L^2$ norm of $f$ . 
Therefore $f=0$ almost everywhere. 
 
\end{proof}
\section{Schrödinger Equation}\label{schrö}

Let  $\Delta_{\R}=-\frac{\partial^2}{\partial x^2}$  the Laplace operator on $\R$. Consider the initial value problem for the time-dependent damped Schrödinger equation on $\R$
\begin{align*}
(\Delta_{\R}-C)u(t,x)+i\frac{\partial}{\partial t} u(t,x)&=0,\\
u(0,x)&=f(x),
\end{align*}
where $C\in \R$ is the damping parameter. For $f\in L^2(\R)$ there is a unique solution  $u\in C^{0}(\R,L^2(\R))$ to this equation in the sense of distributions. 
See for instance \cite{rauch2012partial}. 
S.Chanillo \cite{schroe} showed a Hardy type uniqueness theorem for the Schrödinger equation on $\R$ stating that if for some constants $a,b,A,B$ and some $t_0>0$ we have
$$ \lvert f(x)\rvert \leq Ae^{ax^2}$$
and $$\lvert u(t_0,x)\rvert \leq B\cdot e^{b x^2}$$
with $16abt_0>1$
then $f=0$ a.e., which implies $u=0$ a.e. 
Using a similar technique as in the case of the Morgan theorem we now want to generalise this to the case of harmonic manifolds with pinched negative curvature. 
First consider the Schrödinger equation on $(X,g)$ $u:\R\times X\to \C$
\begin{align*}
   \Delta u(t,x)&= i\frac{\partial}{\partial t}u(t,x)\\
   u(0,x)&=f(x),
\end{align*}
hence $u(t,x)=e^{it\Delta}f(x)$ and we can obtain $u$ for $f\in L^2(X)$ via convolution with the integral kernel $$s_t(x,y)=C_0\int_X e^{-it(\lambda^2+\rho^2)}\varphi_{\lambda,x}(y)\lvert \mathbf{c}(\lambda)\rvert^{-2}\,d\lambda,$$
where $C_0$ is the constant from the Fourier inversion theorem. 
\begin{satz}\label{thm:schroe} 
Let $X$ be a non-compact simply connected harmonic manifold with pinched negative curvature. 
Suppose that for some $\sigma\in X$ we have constants $\alpha,\beta,A,B>0$ such that for $f:X\to \C$ a measurable function:
$$
\lvert f(x)\rvert\leq A\cdot e^{-\alpha d^2_{\sigma}(x)}$$
for all $x\in X$ 
and
$$\lvert u(t_0,x)\rvert\leq B\cdot e^{-\beta d^2(\sigma,x)}$$
for all $x\in X$ and some $t_0>0$,
where $u$ is a solution of the Schrödinger equation on $X$ with initial condition $f$.

Then  $f\in L^2(X)$ and if $16\alpha\beta t_0^2>1$ then $f=0$ almost everywhere and $u=0$ almost everywhere.
\end{satz}

\begin{proof}
First, observe that the bounds on the functions $f$ and $u_{t_0}$ imply that they are in $L^1(X) \cap L^2(X)$. Choose $0<a<\alpha$ and $0<b<\beta$ such that $16ab t_0^2>1$. Then by Proposition \ref{prop:boundradon}
\begin{align*}
\mathcal{R}_{\sigma}(f)(s,\xi)&\leq E_1 e^{-as^2}\\
\mathcal{R}_{\sigma}(u_{t_0})&\leq E_2 e^{-bs^2}
\end{align*}
for some constants  $E_1, E_2>0$.
Since for $g\in L^1(X)$  $$\mathcal{F}(\mathcal{R}_{\sigma}(g))=\widetilde{g}^{\sigma},$$ we have that for every $\xi\in\partial X$  the Radon transform $\mathcal{R}_{\sigma}(u_{t})(\cdot,\xi)$ solves the damped Schrödinger equation on  $\R$ with initial conditions $ \mathcal{R}_{\sigma}(f)(\cdot,\xi)$ and damping parameter $\rho^2$. Hence, the Hardy theorem on the real line yields $\mathcal{R}_{\sigma}(f)(\cdot,\xi)=0$ a.e. and we conclude, using Remark \ref{lemma1} and the Plancherel theorem that $f=0$ a.e. Hence, $u=0$ a.e.
\end{proof}

\section{Hausdorff-Young inequalety}\label{HY}
The Hausdorff-Young inequality for the real line is quite an elementary result only requiring the Plancherel theorem and interpolation arguments provided by the Riesz-Thorin theorem. In the case of generalising it to rank one harmonic manifolds, we must be careful since we are no longer dealing with a linear operator. 
The following is a generalisation of the result of \cite[Theorem 4.2]{interpolation2}. Note that there is a technical error in their proof concerning exactly the point of caution mentioned above. Still, the authors provided a fix for this in private communications in the form of the references \cite{interpolation1} or \cite{Janson1982-1983}.
\begin{satz}\label{thm:Hausdorf-young}
Let $X$ be a non-compact simply connected harmonic manifold of rank one $\partial X$ the visibility boundary, $\{d\mu_{y}\}_{y\in X}$ the family of visibility measures, $1<p<2$, $q$  Hölder conjugated to $p$, $\sigma\in X$ and $C_0$ the constant making the Fourier transform on $X$ an isometry between $L^2(X)$ and $L^2((0,\infty)\times\partial X,d\mu_{\sigma}(\xi)C_0\lvert \mathbf{c}(\lambda) \rvert^{-2}\,d\lambda)$ spaces. Then
\begin{align*}
\Bigl( C_0 \int_{0}^{\infty}\lVert \tilde{f}^{\sigma}(\lambda,\cdot)\rVert^{q}_{L^1(\partial X,d\mu_{\sigma})}\lvert \mathbf{c}(\lambda)\rvert^{-2}\,d\lambda\Bigr)^{\frac{1}{q}}\leq\lVert f \rVert_p.
\end{align*}
\end{satz}

\begin{proof}[ Proof of Theorem \ref{thm:Hausdorf-young}]
 By Lemma \ref{lemma:holo2} the Fourier transform $\tilde{f}^{\sigma}(\lambda,\xi)$ exists for almost every $\xi\in\partial X$ and every $\lambda\in\ (0,\infty)$.
 Hence, for $1\leq p'\leq 2$ and $q'$ Hölder conjugated to $p'$,
 \begin{align*}
 T:L^{p'}(X)&\to L^{q'}((0,\infty),C_0\lvert \mathbf{c}(\lambda)\rvert^{-2}\,d\lambda)\\
 f&\mapsto \lVert \tilde{f}^{\sigma}(\lambda,\cdot)\rVert_{L^1(\partial X,d\mu_{\sigma})}
 \end{align*}
 is a well-defined sublinear operator.\\
 Fix $\lambda\in [0,\infty)$. Then:
 \begin{align*}
 \lVert \tilde{f}^{\sigma}(\lambda,\cdot)\rVert_{L^1(\partial X,d\mu_{\sigma})}&=\int_{\partial X} \lvert \tilde{f}^{\sigma}(\lambda,\xi)\rvert\,d\mu_{\sigma}(\xi)\\
 &\leq \int_{\partial X} \Bigl\lvert\int_X f(x)\cdot e^{-i\lambda-\rho B_{\xi,\sigma}(x)}\,dx\Bigr\rvert\,d\mu_{\sigma}(\xi)\\
 &\leq \int_X \lvert f(x)\rvert \int_{\partial X} \lvert e^{-i\lambda-\rho B_{\xi,\sigma}(x)}\rvert\,d\mu_{\sigma}(\xi)\,dx\\
 &\leq \int_X \lvert f(x)\rvert \varphi_{0,\sigma}(x)\,dx\\
 &\leq \int_X \lvert f(x)\rvert\,dx\\
 &=\lVert f\rVert_{1},
 \end{align*}
 where the last inequality is due to Lemma \ref{lemma:bounds}.
 Hence:
 \begin{align*}
 \lVert T\rVert_{L^1(X)\to L^{\infty}((0,\infty),C_0\lvert\mathbf{c}(\lambda)\rvert^{-2}\,d\lambda)}\leq 1.
 \end{align*}
 Furthermore, using the the Plancherel theorem (Theorem \ref{Placherel Theorem}) we have
 \begin{align*}
 \lVert Tf\rVert^2_{L^{2}([0,\infty),C_0\lvert\mathbf{c}(\lambda)\rvert^{-2}\,d\lambda)}&=C_0\int_0^{\infty}\lVert \tilde{f}^{\sigma}(\lambda,\cdot)\rVert^2_{L^1(\partial X,d\mu_{\sigma})}\lvert\mathbf{c}(\lambda)\rvert^{-2}\,d\lambda\\
 &\leq C_0\int_0^{\infty}\lvert \tilde{f}^{\sigma}(\lambda,\xi)\rvert^2\,d\mu_{\sigma}(\xi)\lvert\mathbf{c}(\lambda)\rvert^{-2}\,d\lambda\\
 &=\lVert \tilde{f}^{\sigma}\rVert^2_{2}\\&=\lVert f\rVert^2_{2}. 
 \end{align*}
 Therefore, 
 \begin{align*}
 \lVert T\rVert_{L^2(X)\to L^{2}((0,\infty),C_0\lvert\mathbf{c}(\lambda)\rvert^{-2}\,d\lambda)}\leq 1.
 \end{align*}
 Hence, by the version of the Riesz interpolation theorem for sublinear operators \cite{interpolation1} or \cite{Janson1982-1983}, we get for $1<p<2$
\begin{align*}
 \lVert T\rVert_{L^p(X)\to L^{q}((0,\infty),C_0\lvert\mathbf{c}(\lambda)\rvert^{-2}\,d\lambda)}\leq 1.
 \end{align*}
 This yields the claim. 
\end{proof}

\section{Hömander theorem}\label{s10}
The Hömander theorem: 
\begin{satz}[\cite{MR1150375}]
If a function $f\in L^1(\R)$ satisfies 
$$\int_\R\int_\R \lvert f(x)\rvert \lvert \mathcal{F}(f)(\xi)\rvert e^{\lvert \xi\rvert \lvert x\rvert}\,dx\,d\xi <\infty.$$
Then $f=0$ almost everywhere,

where $\mathcal{F}(f)(\xi)=\int_\R f(x)e^{-i\xi x}\,dx$ is the euclidean Fourier transform. 
\end{satz}
states that a function and its Fourier transform can not be simultaneously very small at infinity. The same statement holds on Damek-Ricci spaces \cite{MR2679708}.
We now want to generalise this to harmonic manifolds of rank one:

\begin{satz}
Let $f\in L^p(X)$, $1\leq p \leq 2$ if for almost every $\xi\in \partial X$: 
\begin{align}\label{eq:hoe}
\int_\R\int_X \lvert f(x)\rvert \lvert \widetilde{f}^{\sigma}(\lambda,\xi)\rvert e^{(\pm \lvert \lambda \rvert -\rho)B_{\xi,\sigma}(x)}\,dx \lvert \mathbf{c} (\lambda)\rvert^{-2}\,d\lambda <\infty 
\end{align}
then $f=0$ almost everywhere. 
\end{satz}
\begin{proof}
First we have to show that $f\in L^1(X)$: By (\ref{eq:hoe})
$$\int_X \lvert f(x)\rvert e^{(\pm\lambda -\rho)B_{\xi,\sigma}(x)}\,dx<\infty,$$
for almost every $\lambda\in \R$ and $\xi\in \partial X$. 
We can split $X=B_1\cup B_2$, 
where $B_1:=\{x\in X\mid B_{\xi,\sigma}(x)\geq 0\}$ and $B_2=\{x\in X\mid B_{\xi,\sigma}(x)<0\}. $
Choose $\lambda_1$ such that $\lvert \lambda_1 \rvert -\rho >0$ and $\lambda_2$ such that $\lvert \lambda_2\rvert -\rho <0$ then: 
\begin{align*}
\int_X\lvert f(x)\rvert \,dx \leq & \int_{B_1}\lvert f(x)\rvert \,dx +\int_{B_2}\lvert f(x)\rvert \,dx \\
\leq & \int_{B_1}\lvert f(x)\rvert e^{(\lvert \lambda_1\lvert-\rho)B_{\xi,\sigma}(x)}\,dx +\int_{B_2}\lvert f(x)\rvert e^{(\lvert \lambda_2\rvert-\rho)B_{\xi,\sigma}(x)}\,dx\\
<&\infty
\end{align*}
Now we want to deduce the theorem from the classical Hördmander theorem. 
Using horospherical coordinates: 
\begin{align*}
\infty &>\int_{\R}\int_X \lvert f(x)\rvert \lvert \widetilde{f}^{\sigma}(\lambda,\xi)\rvert e^{(\pm \lvert \lambda\rvert -\rho)B_{\xi,\sigma}(x)}\,dx \lvert \mathbf{c}(\lambda)\rvert^{-2}\,d\lambda\\ &=\int_{\R}\int_{\R}\int_{H^s_{\xi,\sigma}} \lvert f(x)\rvert \lvert \widetilde{f}^{\sigma}(\lambda,\xi)\rvert e^{(\pm\lvert \lambda\rvert -\rho)s}\,dx\,ds \lvert \mathbf{c}(\lambda)\rvert^{-2}\,d\lambda \\
&=\int_{\R}\int_{\R}\mathcal{R}(\lvert f\rvert)(s,\xi)\lvert \widetilde{f}^{\sigma}(s,\xi)e^{\pm \lvert \lambda \rvert s}\,ds \lvert \mathbf{c}(\lambda)\rvert^{-2}\,d\lambda.
\end{align*}
Since changing the direction of the ray flips the sing in $\lambda$ and $s$ we obtain:
$$
\int_{\R}\int_{\R}\lvert \mathcal{R}( f)(s,\xi)\rvert \lvert \widetilde{f}^{\sigma}(s,\xi)e^{\lvert \lambda \rvert \lvert s\rvert }\,ds \lvert \mathbf{c}(\lambda)\rvert^{-2}\,d\lambda<\infty.
$$
We conclude since all functions are continuous in $\lambda$: 

$$\int_{\R}\int_{\R}\lvert \mathcal{R}( f)(s,\xi)\rvert \lvert \widetilde{f}^{\sigma}(s,\xi)\rvert e^{\lvert \lambda \rvert \lvert s\rvert }\,ds \,d\lambda<\infty.$$
Hence the claim follows from the classical Hörmander theorem since the Helgason Fourier transform is the Euclidean Fourier transform of the Radon transform. 
\end{proof}

\begin{folg}
Let $f\in L^p(X)$ $1\leq p<2$ if

$$ \int_{\R}\int_X \lvert f(x)\rvert \lvert \widetilde{f}^{\sigma}(\lambda,\xi)\rvert e^{(\lvert \lambda \rvert -\rho)d(\sigma,x)}\,dx  \lvert \mathbf{c}(\lambda)\rvert^{-2}\,d\lambda< \infty, $$
for almost every $\xi\in \partial X$. Then $f=0$ almost everywhere. 

\end{folg}

\begin{proof}
We only need to show that this already implies the condition of the theorem above. 
But we have $\lvert B_{\xi,\sigma}(x)\rvert \leq d(\sigma ,x)$, hence
$$ e^{(\pm\lvert \lambda \rvert -\rho)B_{\xi,\sigma}(x)}\leq e^{(\lvert \lambda \rvert -\rho)\lvert B_{\xi,\sigma}(x)\rvert}\leq e^{(\lvert \lambda \rvert -\rho)d(\sigma,x)}$$

yielding the claim. 
\end{proof}


\footnotesize
\bibliography{literature}
\bibliographystyle{alpha} 


\end{document}